\documentclass[journal]{IEEEtran}
\AtBeginDvi{} 
\usepackage{amssymb}
\usepackage{color}
\usepackage[pdftex]{graphicx}
\usepackage{framed}
\usepackage{bm}
\usepackage{comment}
\usepackage[verbose]{cite}
\usepackage{algorithm}
\usepackage{algorithmic}
\usepackage{slashbox}
\newtheorem{theorem}{Theorem}

\newtheorem{remark}{Remark}
\newtheorem{lemma}{Lemma}
\newtheorem{corollary}{Corollary}

\newtheorem{proof}{Proof}

\newcommand{\maximize}{\mathop{\rm maximize}\limits}

\newcommand{\argmin}{\mathop{\rm argmin}\limits}
\newcommand{\bb}{\mathbb}
\usepackage{latexsym}
\def\qed{\hfill $\Box$} 

\usepackage{ascmac}

\ifCLASSINFOpdf
\else
\fi
\usepackage{amsmath}
\begin{document}
%
\title{Controllability maximization of large-scale systems using projected gradient method}
%
%
%

\author{Kazuhiro~Sato and Akiko~Takeda
\thanks{K. Sato is with the Department of Mathematical Informatics, 
Graduate School of Information Science and Technology, The University of Tokyo,
 Tokyo 113-8656, Japan,
email: kazuhiro@mist.i.u-tokyo.ac.jp}
\thanks{A. Takeda is with the Department of Creative Informatics,
Graduate School of Information Science and Technology,
The University of Tokyo, Tokyo 113-8656, Japan, and RIKEN Center for
Advanced Intelligence Project, 1-4-1, Nihonbashi, Chuo-ku,
Tokyo 103-0027, Japan,
email: takeda@mist.i.u-tokyo.ac.jp}
}

\maketitle
\thispagestyle{empty}
\pagestyle{empty}

\begin{abstract}
In this work, we formulate two controllability maximization problems for large-scale networked dynamical systems such as brain networks:
The first problem is a sparsity constraint optimization problem with a box constraint.
The second problem is a modified problem of the first problem, in which the state transition matrix is Metzler.
In other words, the second problem is a realization problem for a positive system.
We develop a projected gradient method for solving the problems,
and prove global convergence to a stationary point with locally linear convergence rate.
The projections onto the constraints of the first and second problems are given explicitly.
Numerical experiments using the proposed method provide non-trivial results.
In particular, the controllability characteristic is observed to change with increase in the parameter specifying sparsity,
and the change rate appears to be dependent on the network structure.
\end{abstract}


%
\IEEEpeerreviewmaketitle

\section{Introduction} \label{sec1}
%
%
%
%
\IEEEPARstart{C}{ontrollability}, which refers to the possibility to change the present network state to a desired state is a fundamental concept in large-scale networked dynamical systems \cite{gu2015controllability, karrer2020practical, kim2018role, li2017fundamental, liu2011controllability, nepusz2012controlling, olshevsky2014minimal, yan2017network, yuan2013exact}.
Although several real networks are nonlinear \cite{liu2016control, tang2018colloquium, whalen2015observability, zanudo2017structure}, there are no general principles to determine control inputs for a general nonlinear system to reach a desired state.
Moreover, in the context of neural systems, the assumption that a networked system is linear is reasonable,
because the baseline firing rates of the neurons pertaining to the case with no stimulation are considerably lower than those during stimulation and far from the saturation case \cite{galan2008network}.
Consequently, several researchers \cite{gu2015controllability, karrer2020practical, kim2018role, li2017fundamental, liu2011controllability, nepusz2012controlling, olshevsky2014minimal, yan2017network, yuan2013exact} considered the following linear system that is an approximation around an equilibrium point of a nonlinear system:
\begin{align}
\dot{x}(t) = Ax(t) + Bu(t), \label{system}
\end{align}
where $x(t)\in {\bb R}^n$ and $u(t)\in {\bb R}^m$ denote the state and input vectors, respectively; $A\in {\bb R}^{n\times n}$ is a fixed constant matrix that reflects a network structure; and $B\in {\bb R}^{n\times m}$ can be designed to suit the system requirements.
In other words, we can adjust the influence of input $u$ to the autonomous dynamical networked system
$\dot{x}(t)=Ax(t)$ by introducing a decision variable $B$.

The authors in \cite{clark2017submodularity, romao2018distributed, summers2016submodularity} considered
a matrix $B_S$ instead of any $B$,
where $S\subset \{1,2,\ldots, M\}$ specifies the vectors from $M (\geq m)$ candidate column vectors of $B$,
and addressed the following combinatorial problem.
\begin{align}
\maximize_{S\subset \{1,2,\ldots, M\},\, |S|=m}\quad f(B_S), \label{pro1}
\end{align}
where $f(B_S)$ denotes an index of controllability.
In particular, in \cite{clark2017submodularity, summers2016submodularity}, problem \eqref{pro1} was solved based on submodular optimization,
whereas in \cite{romao2018distributed}, problem \eqref{pro1} was transformed into a linear optimization problem and then solved using a primal-dual distributed algorithm.  
The combinatorial problems related to \eqref{pro1} were considered in \cite{pasqualetti2014controllability, tzoumas2016minimal}.
In addition, the authors of \cite{ikeda2018sparsity} considered a controllability maximization problem in the form of a time-varying actuator problem.
Specifically, $\mathcal{B}V(t)$ was considered as the matrix $B$,
where $\mathcal{B}\in {\bb R}^{n\times m}$ and $V(t)\in \{0,1\}^{m\times m}$ denote a fixed constant matrix and time-varying diagonal matrix, respectively,
and an optimization problem to determine the diagonal entries of $V(t)$ was examined.  
This optimization problem was then completely resolved, as reported in \cite{olshevsky2019relaxation}.
However, the methods to determine the $M$ candidate column vectors of $B$ in \cite{clark2017submodularity, romao2018distributed, summers2016submodularity} and the matrix $\mathcal{B}$ in \cite{ikeda2018sparsity, olshevsky2019relaxation} for a large-scale networked dynamical system remain unclear.

To overcome this limitation, in this work, we consider the controllability maximization problems from a different perspective than those considered in \cite{clark2017submodularity, romao2018distributed, summers2016submodularity, ikeda2018sparsity, olshevsky2019relaxation}.
In particular, we formulate two optimal decision problems of $B$ in system \eqref{system} without using candidates of column vectors considered in \cite{clark2017submodularity, romao2018distributed, summers2016submodularity} and a fixed $\mathcal{B}$ considered in \cite{ikeda2018sparsity, olshevsky2019relaxation}.
The first problem is a sparsity constraint optimization problem with a box constraint that specifies the possible values of each element in $B$.
The second problem is a modified problem of the first problem when $A$ is Metzler, i.e., the off-diagonal elements of $A$ are nonnegative.
Specifically, the possible values of each element in $B$ in the second problem are restricted to nonnegative values.
The second problem is thus a realization problem for a positive system, which is an important problem in the control community \cite{haddad2010nonnegative, sato2020construction}.

The contributions of this work can be summarized as follows.
\begin{itemize}
\item Although the considered problems are essentially combinatorial problems,
we solve the problems using a continuous optimization approach. 
In particular, we propose a simple projected (that is, proximal) gradient method 
to solve the considered problems.
Moreover, we demonstrate the global convergence of the proposed algorithm to a stationary point of the objective function.
Furthermore, considering a result recently reported in \cite{li2018calculus}, it is shown that the convergence rate of the sequences generated by the proposed algorithm is locally linear.

\item The projections onto the constraints of the first and second problems are given explicitly.

\item Numerical experiments using the proposed algorithm provide non-trivial results.
In particular, the controllability characteristic is observed to change with increase in the parameter specifying sparsity,
and the change rate appears to be dependent on the network structure that determines the structure of the matrix $A$.
\end{itemize}

The remaining paper is organized as follows.
The problem formulation is described in Section \ref{Sec2}.
Section \ref{Sec3} describes
the
projected gradient methods used to solve the problems and provides
a proof of global convergence to a stationary point with a convergence rate.
Section \ref{Sec4} describes the experimental results, and
the conclusions are presented in Section \ref{Sec5}.

{\it Notation:} The set of real numbers is denoted by ${\bb R}$.
Given matrices $A, B\in {\bb R}^{n\times m}$, we define $\langle A,B\rangle$ and $\|A\|_F$ as the Euclidean inner product and the Frobenius norm, respectively; i.e.,
$\langle A,B\rangle :={\rm tr}(A^{\top}B)$ and
$\|A\|_F:=\sqrt{\langle A,A\rangle}$,
where
the superscript $\top$ denotes the transpose and ${\rm tr}(A)$ denotes the sum of the diagonal elements of $A$.
$\|A\|_0$ is defined as $l^0$ norm; i.e., $\|A\|_0$ denotes the number of nonzero elements in $A$.
For any matrices $A, B\in {\bb R}^{n\times m}$, we write $A\geq B$ $(A\leq B)$ if all the elements of $A$ are greater (less) than or equal to those of $B$.
The symbol ${\bf E}\in {\bb R}^{n\times m}$ denotes a matrix whose elements are only $1$.
The symbol $I_n\in {\bb R}^{n\times n}$ denotes the identity matrix. 


\section{Problem settings} \label{Sec2}

System \eqref{system} is termed controllable \cite{kalman1960general, kalman1963mathematical} if for any desired final state $x_T$ at any desired final time $T$, there exists an input $u$ such that 
$x_T = \int_0^T \exp(A(T-t))Bu(t)dt$.
That is, there exists an input $u$ satisfying $x(0)=0$ and $x(T)=x_T$.
However, the input $u$ might be required to have a high energy.
In other words, even if system \eqref{system} is controllable, it may be difficult to control the system state in practice.
Thus, it is important to consider a degree of controllability.
In this section, we formulate two controllability maximization problems using a controllability index.

First, we consider a general case that $A$ in \eqref{system} is a fixed matrix that may be unstable,
and the controllability Gramian
\begin{align*}
\mathcal{C}_T(B):= \int_0^{T} \exp(A(T-t))BB^{\top}\exp(A^{\top}(T-t))dt.
\end{align*}
The controllability Gramian can be related to the minimum-energy control problem
\begin{align*}
&{\rm minimize}\quad \int_0^T \|u(t)\|^2 dt\\
&{\rm subject\,\,to}\,\,\,\, \eqref{system},\,\,x(0)=0,\,\,x(T)=x_T,
\end{align*}
where $x_T\in {\bb R}^n$ is any final state.
In fact, if system \eqref{system} is controllable, the minimum energy, i.e., the optimal objective value of the above problem, is given by 
$x_T^{\top}\mathcal{C}^{-1}_T(B)x_T$, 
as shown in \cite{liu2016control}. Moreover, the Rayleigh-Ritz theorem implies that
\begin{align}
\frac{1}{\lambda_{\rm max}(\mathcal{C}_T(B))}\leq \frac{x_T^{\top}\mathcal{C}^{-1}_T(B)x_T}{x_T^{\top}x_T}\leq \frac{1}{\lambda_{\rm min}(\mathcal{C}_T(B))} \label{Rayleigh}
\end{align}
for any $x_T\neq 0$, where $\lambda_{\rm min}(\mathcal{C}_T(B))$ and $\lambda_{\rm max}(\mathcal{C}_T(B))$ denote the minimum and maximum eigenvalues of $\mathcal{C}_T(B)$, respectively.
Because inequality \eqref{Rayleigh} holds, $\lambda_{\rm min}(\mathcal{C}_T(B))$ and ${\rm tr}(\mathcal{C}_T^{-1}(B))$, which denotes the sum of the inverse of all the eigenvalues of $\mathcal{C}_T(B)$, are frequently adopted as the controllability indices \cite{gu2015controllability, pasqualetti2014controllability, summers2016submodularity, wu2018benchmarking}.
Moreover,
${{\rm tr}(\mathcal{C}_T^{-1}(B))} > \frac{n}{{\rm tr} (\mathcal{C}_T(B))}$.
Thus, to decrease ${\rm tr}(\mathcal{C}_T^{-1}(B))$, ${\rm tr} (\mathcal{C}_T(B))$, which is the sum of all eigenvalues of $\mathcal{C}_T(B)$, must be increased.
If ${\rm tr} (\mathcal{C}_T(B))$ is sufficiently large, there exists a direction that can be specified by an eigenvector of $\mathcal{C}_T(B)$ such that the state $x(t)$ of system \eqref{system} 
can move to
the direction with a low input energy.
Consequently, in the existing studies
\cite{romao2018distributed, ikeda2018sparsity, olshevsky2019relaxation}, ${\rm tr} (\mathcal{C}_T(B))$ was adopted as a controllability index.
Note that unlike ${\rm tr}(\mathcal{C}_T^{-1}(B))$, ${\rm tr} (\mathcal{C}_T(B))$ can be defined even if system \eqref{system} is not controllable.

In this study, we use ${\rm tr} (\mathcal{C}_T(B))$ as a controllability index and
consider the following modified problems:
\begin{align}
&{\rm minimize}\quad h(B):=-{\rm tr}(\mathcal{C}_T(B)) \label{problem0}\\
&{\rm subject\,\,to}\,\,\,\, \|B\|_0 \leq s,
\nonumber
\end{align}
where $s$ is a specified nonnegative value.
The constraint $\|B\|_0 \leq s$ ensures that the nonzero elements of $B$ are less than or equal to $s$.
That is, we can determine sparsity of $B$ by specifying $s\in \{1,2,\ldots, nm\}$.

However, in general, $h(B)$ is not bounded below, subject to $\|B\|_0 \leq s$.
That is, a global optimal solution for \eqref{problem0} does not exist.
To demonstrate this aspect, we consider any $\beta>0$ and $\tilde{B}\in {\bb R}^{n\times m}$.
Because $h(\beta\tilde{B}) = \beta^2 h(\tilde{B})$ and $\|\beta \tilde{B}\|_0 = \|\tilde{B}\|_0$,
$h(\tilde{B})<0$ (this relation holds if system \eqref{system} with $B=\tilde{B}$ is controllable) and $\|\tilde{B}\|_0 \leq s$ imply that $\lim_{\beta\rightarrow \infty} h(\beta \tilde{B})=-\infty$ and $\lim_{\beta\rightarrow \infty} \|\beta \tilde{B}\|_0 \leq s$.

To guarantee the existence of an optimal solution, we consider the following problem with a box constraint.

\begin{framed}
\noindent
Problem 1: Given any matrix $A\in {\bb R}^{n\times n}$, $T>0$, and $s\in \{1,2,\ldots, nm\}$, find $B\in {\bb R}^{n\times m}$ that solves
\vspace{-0.5em}
\begin{align}
&{\rm minimize}\quad h(B) \nonumber\\
&{\rm subject\, to}\quad \|B\|_0 \leq s,\,\, -{\bf E}\leq B\leq {\bf E}.\nonumber
\end{align}
\vspace{-2em}
\end{framed}

The above discussion implies that
if $B^*$ is an optimal solution to Problem 1, $\beta B^*$ is that to a modified problem in which $-{\bf E}\leq B\leq {\bf E}$ in Problem 1 is replaced with $-\beta{\bf E}\leq B\leq \beta{\bf E}$,
where $\beta>0$.

Next, we consider a positive case in which $A$ in \eqref{system} is a fixed Metzler matrix that may be unstable.
Positive systems with Metzler matrix $A$ and non-negative matrix $B$ are important, as mentioned in Section \ref{sec1}.
To realize a positive system when $A$ is Metzler, we consider the following problem.

\begin{framed}
\noindent
Problem 2: Given any Metzler matrix $A\in {\bb R}^{n\times n}$, $T>0$, and $s\in \{1,2,\ldots, nm\}$, find $B\in {\bb R}^{n\times m}$ that solves
\vspace{-0.5em}
\begin{align}
&{\rm minimize}\quad h(B) \nonumber\\
&{\rm subject\, to}\quad \|B\|_0 \leq s,\,\, 0\leq B\leq {\bf E}.\nonumber
\end{align}
\vspace{-2em}
\end{framed}

That is, we replace the constraint $-{\bf E}\leq B\leq {\bf E}$ in Problem 1 with $0\leq B\leq {\bf E}$.

\begin{remark} \label{remark1}
The $l^0$ norm constraint is frequently replaced with the $l^1$ norm constraint when an objective function and other constraints are convex,
because the modified problem then becomes convex \cite{candes2005decoding, tibshirani1996regression}.
However, our objective function $h$ is non-convex, as shown in Section \ref{Sec3}.
Thus, even if we replace $l^0$ with $l^1$, the modified problems are non-convex.
Hence, in this study, we do not replace $l^0$ with $l^1$.
\end{remark}

\begin{remark} \label{remark2}
The objective values at the global optimal solutions to Problems 1 and 2 monotonically increase as $s$ increases.
That is, when we use the global optimal solutions, controllability index $-h(B)$ increases as $s$ increases.
\end{remark}


\section{Projected gradient methods for Problems 1 and 2}  \label{Sec3}

In this section, we develop projected gradient methods for solving Problems 1 and 2.
To this end, we consider 
\begin{align}
&{\rm minimize}\quad \mathcal{I}_{Z_s \cap W}(B) + h(B), \label{problem1} \\
&{\rm minimize}\quad \mathcal{I}_{Z_s \cap W_{[0,1]}}(B) + h(B), \label{problem2}
\end{align}
where \eqref{problem1} and \eqref{problem2} are unconstrained problems equivalent to Problems 1 and 2, respectively, 
\begin{align*}
\mathcal{I}_{\mathcal{S}}(B):= \begin{cases}
0\quad\,\,\, (B\in \mathcal{S})\\
\infty\quad (B\not\in \mathcal{S})
\end{cases}
\end{align*}
is the indicator function of any set $\mathcal{S}$, and
\begin{align*}
Z_s &:= \{ B\in {\bb R}^{n\times m}\,|\, \|B\|_0\leq s\}, \\
W &:= \{ B\in {\bb R}^{n\times m}\,|\, -{\bf E}\leq B\leq {\bf E}\}, \\
W_{[0,1]} &:= \{ B\in {\bb R}^{n\times m}\,|\, 0\leq B\leq {\bf E}\}.
\end{align*}

We first note that
\begin{align}
h(B) &= - \int_0^T {\rm tr}(\exp(A(T-t))BB^{\top}\exp(A^{\top}(T-t))) dt \nonumber\\
&= -\int_0^T {\rm tr}(BB^{\top}\exp(A^{\top}(T-t))\exp(A(T-t))) dt \nonumber\\
&= {\rm tr}(BB^{\top}H(A,T)), \label{lem_h}
\end{align}
where 
\begin{align*}
H(A,T):=-\int_0^T \exp (A^{\top}(T-t)) \exp(A(T-t)) dt.
\end{align*}
Using \eqref{lem_h}, we prove the following theorem.

\begin{theorem} \label{thm1}
The function $h(B)$ is strictly concave, and the gradient is
\begin{align}
\nabla h(B) = 2H(A,T) B. \label{nabla_h}
\end{align}
\end{theorem}
\begin{proof}
From \eqref{lem_h},
the directional derivative of $h$ at $B$ along $B'$ is given by
${\rm D}h(B)[B'] = 2{\rm tr} \left(B'^{\top} H(A,T)B\right)$.
Hence, the gradient of $h(B)$ is given by \eqref{nabla_h}.
Also, the Hessian of $h$ at any $B\in {\bb R}^{n\times m}$ is given by
${\rm Hess}\,h(B)[B'] = 2H(A,T)B'$.
Because $H(A,T)$ is a symmetric negative definite matrix,
$\langle B', {\rm Hess}\,h(B)[B']\rangle <0$
for any $B'\in {\bb R}^{n\times m}\backslash\{0\}$.
Hence, $h(B)$ is strictly concave. \qed
\end{proof}

\noindent
The following corollary follows from \eqref{nabla_h}.
\begin{corollary} \label{cor1}
The gradient $\nabla h$ is $L(A,T)$-Lipschitz continuous,
where
\begin{align}
L(A,T):=2\int_0^T \|\exp(A(T-t))\|_F^2 dt. \label{LAT}
\end{align}
That is,
\begin{align}
\|\nabla h(B_1)-\nabla h(B_2)\|_F \leq L(A,T) \|B_1-B_2\|_F, \label{Lipschitz}
\end{align}
where $B_1$ and $B_2$ are any real $n\times m$ matrices.
\end{corollary}

Algorithm 1 is the proposed algorithm for solving Problems 1 and 2.
In practice, we terminate the iteration if $\|\nabla h(B_k)\|_F$ is sufficiently small.
Note that we must choose a nonzero $B_0$ at step 1.
This is because it follows from \eqref{nabla_h} that $B_0=0$ implies $\nabla h(B_0)=0$.

\begin{algorithm} 
\caption{Projected gradient methods for Problem 1 and 2.}         
\begin{algorithmic}[1]
\STATE Set $B_0\in \bb{R}^{n\times m}\backslash\{0\}$ and $t>L(A,T)$, where $L(A,T)$ is defined as \eqref{LAT}.
\FOR{$k=0,1,2,\ldots$ }
\STATE $B_{k+1}\in {\rm P}\left(B_k -\frac{1}{t}\nabla h(B_k)\right)$, where ${\rm P}$ is the projection onto $Z_s \cap W$ for Problem 1 and $Z_s \cap W_{[0,1]}$ for Problem 2.
\ENDFOR
\end{algorithmic}
\end{algorithm}

Using Corollary \ref{cor1}, we can obtain the following theorem regarding global convergence and convergence rate.
To show this and for the following subsections,
we define 
$\Lambda$ as the index set of ${\bb R}^{n\times m}$.
That is, $\Lambda:= \{ (i,j)\,|\, i\in \{1,2,\ldots, n\},\,\,j\in \{1,2,\ldots, m\}\}$.

\begin{theorem}
Any sequence $\{B_k\}$ generated by Algorithm 1 for Problem 1 (Problem 2) globally converges to a stationary point of $\mathcal{I}_{Z_s \cap W}(B) + h(B)$ ($\mathcal{I}_{Z_s \cap W_{[0,1]}}(B) + h(B)$) with locally linear convergence rate.
\end{theorem}
\begin{proof}
We only prove the claim on Problem 1, because the proof on Problem 2 is similar.

First, we show the global convergence.
The objective function of \eqref{problem1}, that is,
$\mathcal{I}_{Z_s \cap W}(B) + h(B)$, is a proper lower semicontinuous KL function \cite{bolte2014proximal}.
Moreover, $\mathcal{I}_{Z_s \cap W}(B) + h(B)$ is bounded below.
In fact,
$B\in Z_s \cap W$ implies $\mathcal{I}_{Z_s \cap W}(B) + h(B)>-\infty$, and $B\not\in Z_s \cap W$ yields $\mathcal{I}_{Z_s \cap W}(B) = \infty > -h(B)$.
Thus, if $B\not\in Z_s \cap W$, then $\mathcal{I}_{Z_s \cap W}(B) + h(B)>-h(B)+h(B)=0$.
From Corollary \ref{cor1}, $\nabla h$ is $L(A,T)$-Lipschitz continuous.
Thus, Theorem 5.1 and Remark 5.2 in \cite{attouch2013convergence} imply global convergence. 

Next, we show locally linear convergence rate.
Let $\Gamma_k$ be the set of all sets of $k$ different indices $(i,j)\in \Lambda$. 
Using $\Gamma_k$ and \eqref{lem_h}, Problem 1 can be rewritten as
\begin{align*}
&\min_B\quad \min_{\Omega\in \Gamma_{nm-s}}\,\, {\rm tr}(BB^{\top}H(A,T)) \nonumber\\
&{\rm subject\, to}\quad B\in H_{\Omega}\cap W.\nonumber
\end{align*}
with $H_{\Omega}:=\{ B\in {\bb R}^{n\times m}\,|\,B_{ij}=0\,\, {\rm for}\,\, (i,j)\in \Omega\}$.
Moreover, this problem is equivalent to
$\min_B\,\, F(B)$,
where $F(B):= \min_{\Omega\in \Gamma_{nm-s}}\mathcal{I}_{H_{\Omega}\cap W}(B)+ {\rm tr}(BB^{\top}H(A,T))$.
Because $F(B)$ is continuous on $\{B\in {\bb R}^{n\times m}\,|\, \partial F(B)\neq \emptyset\}$, Corollary 5.2 in \cite{li2018calculus} implies locally linear convergence rate. \qed
\end{proof}

In the following subsections, we show that the projections onto $Z_s \cap W$ and $Z_s \cap W_{[0,1]}$ in Algorithm 1 can be easily calculated.
To this end, we define
\begin{align*}
\|X\|_{\Gamma} := \sqrt{\sum_{(i,j)\in \Gamma} X_{ij}^2},\quad
\langle X,Y \rangle_{\Gamma} := \sum_{(i,j)\in \Gamma} X_{ij}Y_{ij}
\end{align*}
for any $X, Y\in {\bb R}^{n\times m}$ and any $\Gamma \subset \Lambda$.

\subsection{Projection onto $Z_s \cap W$} \label{Sec3A}

The projections of $B\in {\bb R}^{n\times m}$ onto $Z_s$, $W$, and $Z_s \cap W$ are defined by
\begin{align}
{\rm P}_{Z_s}(B)&:= \argmin_{X\in {\bb R}^{n\times m}} \left\{ \|X-B\|_F^2\,|\, \|X\|_0\leq s\right\}, \nonumber\\
{\rm P}_{W}(B)&:= \argmin_{X\in {\bb R}^{n\times m}} \left\{ \|X-B\|_F^2\,|\,-{\bf E}\leq X\leq {\bf E} \right\}, \nonumber \\
&= \min(\max(-{\bf E}, B), {\bf E})\nonumber\\
{\rm P}_{Z_s\cap W}(B) &:=\argmin_{X\in {\bb R}^{n\times m}} \left\{ \|X-B\|_F^2\,|\, -{\bf E}\leq X\leq {\bf E},\right.\nonumber\\
&\quad\, \left.\,\,\|X\|_0\leq s\right\}, \label{projection1}
\end{align}
respectively.
Note that ${\rm P}_{Z_s}$ can be a set, in general. In fact, $({\rm P}_{Z_s}(B))_{ij} = B_{ij}$ if $B_{ij}$ is contained in the $s$ first largest entries in absolute value of $B$, and $({\rm P}_{Z_s}(B))_{ij}=0$ otherwise.
That is, the set-valuedness of ${\rm P}_{Z_s}$ arises from the fact that the $s$ largest entries may not be uniquely defined.
Also, for any $B\in {\bb R}^{n\times m}$, let $\Lambda_s(B)$ be any index set that indicates the $s$ first largest entries in $|B_{ij}|$, $(i,j)\in \Lambda$.
By definition, for any $B, X\in {\bb R}^{n\times m}$,
\begin{align}
\|X\|_F^2 = \|X\|^2_{\Lambda_s(B)} + \|X\|_{\Lambda\backslash \Lambda_s(B)}^2. \label{hoge}
\end{align}
Furthermore, we have the following lemma.

\begin{lemma} \label{lem1}
For any $B, X\in {\bb R}^{n\times m}$,
\begin{align}
\|X-B\|_{\Lambda_s(B)}^2+\|X\|_{\Lambda\backslash \Lambda_s(B)}^2 = \|X-{\rm P}_{Z_s}(B)\|_F^2. \label{hoge2}
\end{align}
\end{lemma}
\begin{proof}
It follows from \eqref{hoge} that
\begin{align*}
\|X-{\rm P}_{Z_s}(B)\|_F^2 =& \|X-{\rm P}_{Z_s}(B)\|^2_{\Lambda_s(B)}\\
& + \|X-{\rm P}_{Z_s}(B)\|_{\Lambda\backslash\Lambda_s(B)}^2.
\end{align*}
By definition,
\begin{align*}
||X-{\rm P}_{Z_s}(B)||^2_{\Lambda_s(B)} &= ||X-B||^2_{\Lambda_s(B)} \\
\|X-{\rm P}_{Z_s}(B)\|_{\Lambda\backslash\Lambda_s(B)}^2 &= \|X\|_{\Lambda\backslash\Lambda_s(B)}^2.
\end{align*}
Thus, \eqref{hoge2} holds. \qed
\end{proof}

Using \eqref{hoge} and Lemma \ref{lem1},
we can show that
the projection onto $Z_s \cap W$ is the composition of ${\rm P}_{Z_s}$ and ${\rm P}_{W}$.

\begin{theorem} \label{thm_pro1}
\begin{align*}
{\rm P}_{Z_s\cap W} = {\rm P}_{W}\circ {\rm P}_{Z_s}.
\end{align*}
\end{theorem}
\begin{proof}
It follows from \eqref{projection1} and \eqref{hoge} that
\begin{align}
{\rm P}_{Z_s\cap W}(B) &= \argmin_{X\in {\bb R}^{n\times m}} \{\|X\|^2_{\Lambda_s(B)} + \|X\|_{\Lambda\backslash \Lambda_s(B)}^2 \nonumber\\
&\quad -2 \langle B,X \rangle_{\Lambda_s(B)} -2\langle B,X\rangle_{\Lambda\backslash \Lambda_s(B)} + \|B\|_F^2\,| \nonumber\\
&\quad -{\bf E}\leq X\leq {\bf E},\,\,\|X\|_0\leq s \} \label{hoge3}
\end{align}
Because each $|B_{ij}|$, $(i,j)\in \Lambda_s(B)$ is greater than all $|B_{ij}|$, $(i,j)\in \Lambda\backslash \Lambda_s(B)$ and $\|B\|_F^2$ is constant, \eqref{hoge3} is equivalent to
\begin{align}
{\rm P}_{Z_s\cap W}(B) &= \argmin_{X\in {\bb R}^{n\times m}} \{\|X\|^2_{\Lambda_s(B)}-2 \langle B,X \rangle_{\Lambda_s(B)}\,| \nonumber \\
&\quad \|X\|_{\Lambda\backslash \Lambda_s(B)}^2=0,-{\bf E}\leq X\leq {\bf E}, \|X\|_0\leq s \} \nonumber \\
&=\argmin_{X\in {\bb R}^{n\times m}} \{\|X\|^2_{\Lambda_s(B)}-2 \langle B,X \rangle_{\Lambda_s(B)}\,| \nonumber \\
&\quad \|X\|_{\Lambda\backslash \Lambda_s(B)}^2=0,-{\bf E}\leq X\leq {\bf E}\} \nonumber \\
&=\argmin_{X\in {\bb R}^{n\times m}} \{\|X-B\|^2_{\Lambda_s(B)} +\|X\|_{\Lambda\backslash \Lambda_s(B)}^2\,| \nonumber \\
&\quad-{\bf E}\leq X\leq {\bf E}\} \label{hoge4}
\end{align}
From \eqref{hoge2} in Lemma \ref{lem1}, \eqref{hoge4} is equivalent to
\begin{align*}
{\rm P}_{Z_s\cap W}(B) &=
\argmin_{X\in {\bb R}^{n\times m}} \left\{ \|X-{\rm P}_{Z_s}(B)\|_F^2\,|\,-{\bf E}\leq X\leq {\bf E} \right\} \\
&={\rm P}_{W}({\rm P}_{Z_s}(B)).
\end{align*}
This completes the proof. \qed
\end{proof}

In general, 
${\rm P}_{Z_s\cap W} \neq {\rm P}_{Z_s}\circ {\rm P}_{W}$.
That is, the order of the projections ${\rm P}_{Z_s}$ and ${\rm P}_{W}$ is not commutative.
In fact, for example, suppose that $B=\begin{pmatrix}
3 \\
-4
\end{pmatrix}$ and $s=1$.
Then,
$P_W( P_{Z_1}(B)) = \begin{pmatrix}
0 \\
-1
\end{pmatrix}$ and
$P_{Z_1}(P_W(B)) = \begin{pmatrix}
1 \\
0
\end{pmatrix}$.
Thus,
$\| P_W(P_{Z_1}(B)) - B\|_F^2 = 18 < 20 = \|P_{Z_1}(P_W(B))-B\|_F^2$.

\subsection{Projection onto $Z_s \cap W_{[0,1]}$} \label{Sec3B}

The projections of $B\in {\bb R}^{n\times m}$ onto $\{B\in {\bb R}^{n\times m}\,|\, B\geq 0\}$, $\{B\in {\bb R}^{n\times m}\,|\, B\leq {\bf E}\}$, and $Z_s \cap W_{[0,1]}$ are defined by
\begin{align}
{\rm P}_{\geq 0}(B) &:= \argmin_{X\in {\bb R}^{n\times m}} \left\{ \|X-B\|_F^2\,|\, X\geq 0 \right\}, \nonumber\\
&= \max(0,B), \nonumber \\
{\rm P}_{\leq 1}(B) &:= \argmin_{X\in {\bb R}^{n\times m}} \left\{ \|X-B\|_F^2\,|\, X\leq {\bf E} \right\}, \nonumber\\
&= \min({\bf E},B), \nonumber\\
{\rm P}_{Z_s\cap W_{[0,1]}}(B) &:=\argmin_{X\in {\bb R}^{n\times m}} \left\{ \|X-B\|_F^2\,|\, 0\leq X\leq {\bf E},\right.\nonumber\\
&\quad\, \left.\,\,\|X\|_0\leq s\right\}, \label{projection2}
\end{align}
respectively.
For any $B\in {\bb R}^{n\times m}$, the index sets $\Gamma_{<0}(B)$, $\Gamma_{[0,1]}(B)$, and $\Gamma_{>1}(B)$ are defined by
\begin{align*}
\Gamma_{<0}(B) &:= \{(i,j)\in \Lambda\,|\, B<0\},\\
\Gamma_{[0,1]}(B) &:= \{(i,j)\in \Lambda\,|\, 0\leq B\leq {\bf E}\},\\
\Gamma_{>1}(B) &:= \{(i,j)\in \Lambda\,|\, B>{\bf E} \},
\end{align*}
respectively.
By definition, for any $B, X\in {\bb R}^{n\times m}$,
\begin{align}
\|X\|_F^2 = \|X\|^2_{\Gamma_{<0}(B)} + \|X\|_{\Gamma_{[0,1]}(B)}+\|X\|_{\Gamma_{>1}(B)}^2. \label{hoge5}
\end{align}
Also, for any $B\in {\bb R}^{n\times m}$, let $\tilde{\Lambda}_s(B)$ be the index set that indicates the $\min(s,|\tilde{\Lambda}(B)|)$ first largest entries in $|B_{ij}|$, $(i,j)\in \tilde{\Lambda}(B)$, where 
\begin{align*}
\tilde{\Lambda}(B):=\Gamma_{[0,1]}(B)\cup\Gamma_{>1}(B).
\end{align*}

The projection onto $Z_s \cap W_{[0,1]}$ is the composition of ${\rm P}_{\geq 0}$, ${\rm P}_{Z_s}$, and ${\rm P}_{\leq 1}$ as follows.
\begin{theorem} \label{thm_pro2}
\begin{align*}
{\rm P}_{Z_s\cap W_{[0,1]}} = {\rm P}_{\leq 1}\circ {\rm P}_{Z_s}\circ {\rm P}_{\geq 0}.
\end{align*}
\end{theorem}
\begin{proof}
It follows from \eqref{projection2} and \eqref{hoge5} that
\begin{align*}
{\rm P}_{Z_s\cap W_{[0,1]}}(B) &= \argmin_{X\in {\bb R}^{n\times m}} \{\|X-B\|^2_{\Gamma_{<0}(B)} + \|X-B\|_{\Gamma_{[0,1]}(B)}^2 \\
&\quad +\|X-B\|_{\Gamma_{>1}(B)}^2\,|\, 0\leq X\leq {\bf E},\,\|X\|_0\leq s\}.
\end{align*}
If $(i,j)\in \Gamma_{<0}(B)$, then $B_{ij}<0$. Thus, $\|X-B\|_{\Gamma_{<0}(B)}^2=\|X\|_{\Gamma_{<0}(B)}^2 -2\langle B,X\rangle_{\Gamma_{<0}(B)} +\|B\|_{\Gamma_{<0}(B)}^2$ subject to  $0\leq X\leq {\bf E}$, $\|X\|_0\leq s$ is minimized when $\|X\|_{\Gamma_{<0}(B)}=0$.
Thus,
\begin{align}
{\rm P}_{Z_s\cap W_{[0,1]}}(B) &= \argmin_{X\in {\bb R}^{n\times m}} \{\|X-B\|^2_{\tilde{\Lambda}(B)}\,|\, \|X\|_{\Gamma_{<0}(B)}=0, \nonumber\\ 
&\quad 0\leq X\leq {\bf E},\,\|X\|_0\leq s \}. \label{hogehoge1} 
\end{align}
Because $B_{ij}\geq 0$ for any $(i,j)\in \tilde{\Lambda}(B)$, \eqref{hogehoge1} implies
\begin{align}
{\rm P}_{Z_s\cap W_{[0,1]}}(B) &= \argmin_{X\in {\bb R}^{n\times m}} \{\|X-B\|^2_{\tilde{\Lambda}(B)}\,|\, \|X\|_{\Gamma_{<0}(B)}=0, \nonumber \\ 
&\quad X\leq {\bf E},\,\|X\|_0\leq s \}. \label{hogehoge2}
\end{align}
Moreover, because
\begin{align*}
 \|X-B\|_{\tilde{\Lambda}(B)}^2 
=& \|X\|_{\tilde{\Lambda}_s(B)}^2+\|X\|_{\tilde{\Lambda}(B)\backslash \tilde{\Lambda}_s(B)}^2 -2\langle B,X\rangle_{\tilde{\Lambda}_s(B)} \\
&-2\langle B,X\rangle_{\tilde{\Lambda}(B)\backslash \tilde{\Lambda}_s(B)} + \|B\|_{\tilde{\Lambda}(B)}^2
\end{align*}
and each $B_{ij}$ for $(i,j)\in \tilde{\Lambda}_s(B)$ is greater than all $B_{ij}$, $(i,j)\in\tilde{\Lambda}(B)\backslash \tilde{\Lambda}_s(B)$, \eqref{hogehoge2} yields
\begin{align*}
{\rm P}_{Z_s\cap W_{[0,1]}}(B) 
&=\argmin_{X\in {\bb R}^{n\times m}} \{ \|X-B\|^2_{\tilde{\Lambda}_s(B)}+\|X\|_{\Gamma_{<0}(B)}^2\\
&\quad+ \|X\|_{\tilde{\Lambda}(B)\backslash \tilde{\Lambda}_s(B)}^2 \,|\, X\leq {\bf E} \} \\
&=\argmin_{X\in {\bb R}^{n\times m}} \{ \|X-{\rm P}_{Z_s}({\rm P}_{\geq 0}(B))\|^2_{\tilde{\Lambda}_s(B))} \\
&\quad +\|X-{\rm P}_{Z_s}({\rm P}_{\geq 0}(B))\|_{\Gamma_{<0}(B)}^2\\
&\quad+ \|X-{\rm P}_{Z_s}({\rm P}_{\geq 0}(B))\|_{\tilde{\Lambda}(B)\backslash \tilde{\Lambda}_s(B)}^2 \,|\ X\leq {\bf E} \}\\
&=\argmin_{X\in {\bb R}^{n\times m}} \{ \|X-{\rm P}_{Z_s}({\rm P}_{\geq 0}(B))\|^2_{F} \,|\, X\leq {\bf E} \}\\
&={\rm P}_{\leq 1}({\rm P}_{Z_s}({\rm P}_{\geq 0}(B))).
\end{align*}
Here, the first equality follows from a similar discussion to the proof of Theorem 2,
and
 the second equality follows from
\begin{align*}
\|X-{\rm P}_{Z_s}({\rm P}_{\geq 0}(B))\|_{\tilde{\Lambda}_s(B)} &= \|X-B\|_{\tilde{\Lambda}_s(B)}, \\
||X-{\rm P}_{Z_s}({\rm P}_{\geq 0}(B))||_{\Gamma_{<0}(B)} &= ||X||_{\Gamma_{<0}(B)}, \\
\|X-{\rm P}_{Z_s}({\rm P}_{\geq 0}(B))\|_{\tilde{\Lambda}(B)\backslash \tilde{\Lambda}_s(B)} &=\|X\|_{\tilde{\Lambda}(B)\backslash \tilde{\Lambda}_s(B)}.
\end{align*}
This completes the proof. \qed
\end{proof}

Note that, similarly to in Theorem \ref{thm_pro1}, we cannot change the order of
${\rm P}_{\leq 1}$, ${\rm P}_{Z_s}$, and ${\rm P}_{\geq 0}$. In fact, for example, suppose that $B=\begin{pmatrix}
3 \\
-4
\end{pmatrix}$ and $s=1$.
Then,
$P_{\leq 1}( P_{Z_1}(P_{\geq 0}(B))) = \begin{pmatrix}
1 \\
0
\end{pmatrix}$ and
$P_{\geq 0}( P_{Z_1}(P_{\leq 1}(B))) = \begin{pmatrix}
0 \\
0
\end{pmatrix}$.
Thus,
$\| P_{\leq 1}( P_{Z_1}(P_{\geq 0}(B))) - B\|_F^2 = 20 < 25 = \|P_{\geq 0}( P_{Z_1}(P_{\leq 1}(B)))-B\|_F^2$.

\section{Numerical Experiments} \label{Sec4}

This section describes the results of the numerical experiment performed using Algorithm 1.
In all the cases, we set $m=1$, although Algorithm 1 can also be used for $m>1$.
That is, we only considered a single input case
because the case is already sufficiently difficult.
In fact, it has been known that 
the problem of finding $B\in {\bb R}^{n\times 1}$ such that
system \eqref{system} is controllable is NP hard \cite{olshevsky2014minimal}.
Moreover, we set final time $T=10$ and $t=1.1 L(A,T)$ in Algorithm 1.

For Problem 1, we constructed $A$ in system \eqref{system} by using 
MATLAB command ${\rm sprandn}$ (that is a sparse normally distributed random matrix generator) and
the Watts--Strogats model with $n$ nodes, $6$ average degree, and $0.05$ rewiring probability \cite{watts1998collective}.
For Problem 2, we constructed a Metzler matrix $A$ in system \eqref{system} by using MATLAB command ${\rm sprand}$ (that is a sparse uniformly distributed random matrix generator on the interval $(0,1)$) and the Watts--Strogats model that has the same parameters as those used in Problem 1.



Figs.\,\ref{Fig1} and \ref{Fig2}
show the relations between controllability index $-h(B)$ and sparsity parameter $s$ in Problems 1 and 2, respectively.
Here, initial point $B_0$ in Algorithm 1 was the same for all $s$ when $n$ was fixed.
As shown in Fig.\,\ref{Fig1}, $-h(B)$ tended to increase as $s$ increased.
However, $-h(B)$ did not monotonically increase, although $-h(B)$ is higher as $s$ increases when we use global optimal solutions, as mentioned in Remark \ref{remark2}.
This means that local optimal solutions to Problem 1 could be obtained using Algorithm 1.
In contrast, $-h(B)$ monotonically increased as $s$ increased for Problem 2.

\begin{figure}[t]
 \begin{tabular}{cc}
\begin{minipage}{0.45\hsize}
\begin{center}
\includegraphics[width = 4.1cm, height = 2.8cm]{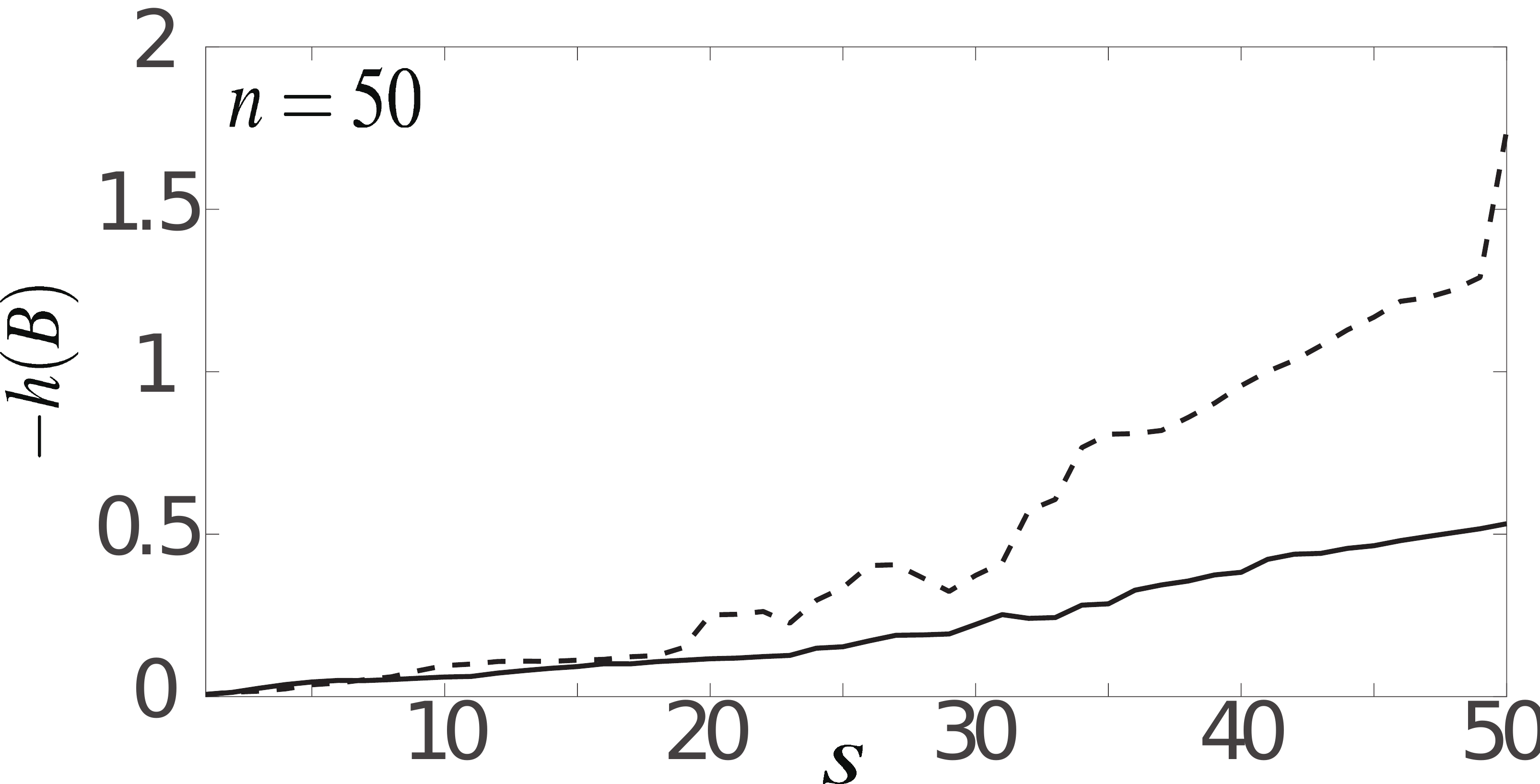}\\
\includegraphics[width = 4.5cm, height = 3cm]{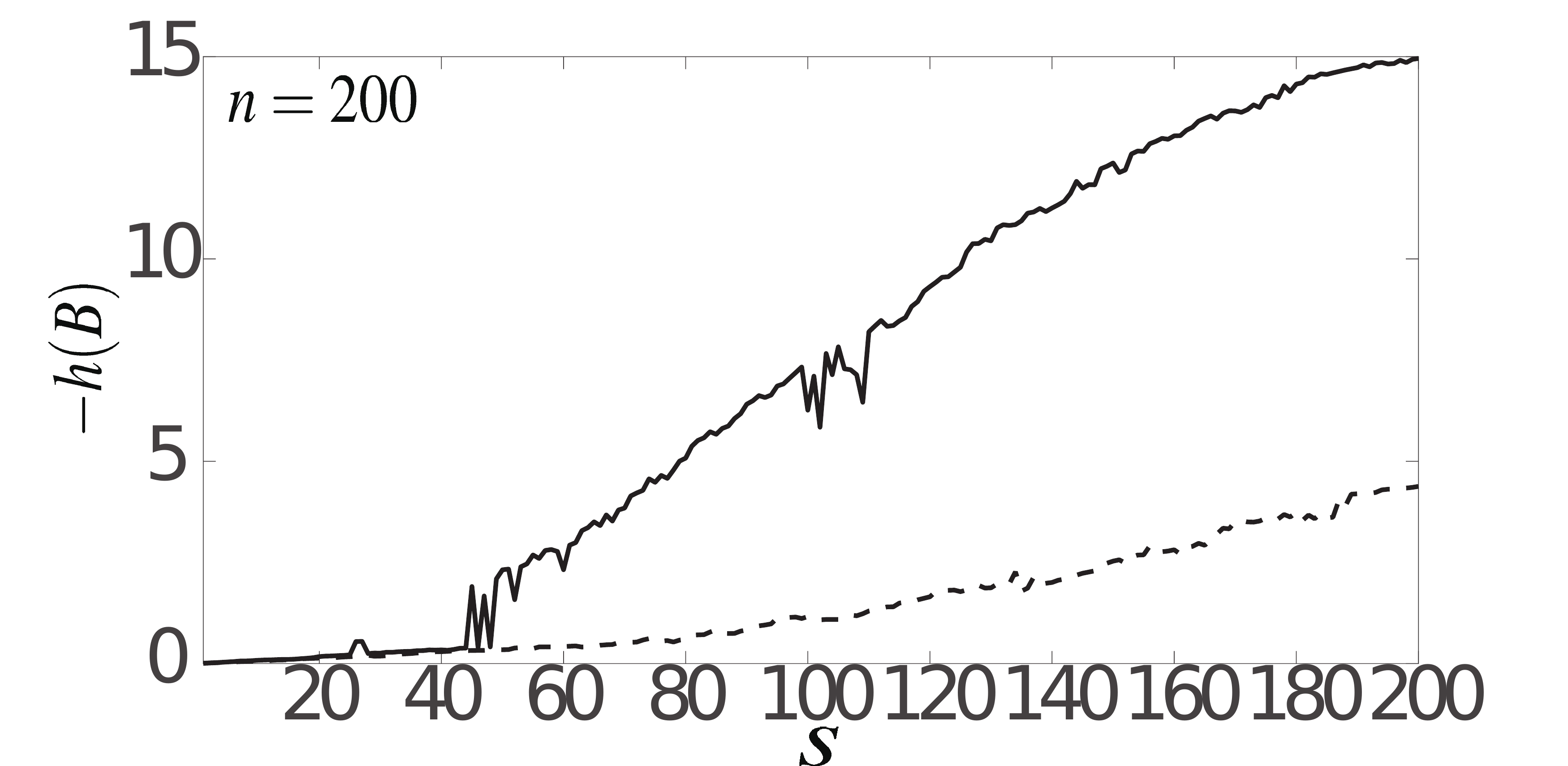}
\end{center}
\end{minipage} &
\begin{minipage}{0.45\hsize}
\begin{center}
\includegraphics[width = 4.5cm, height = 3cm]{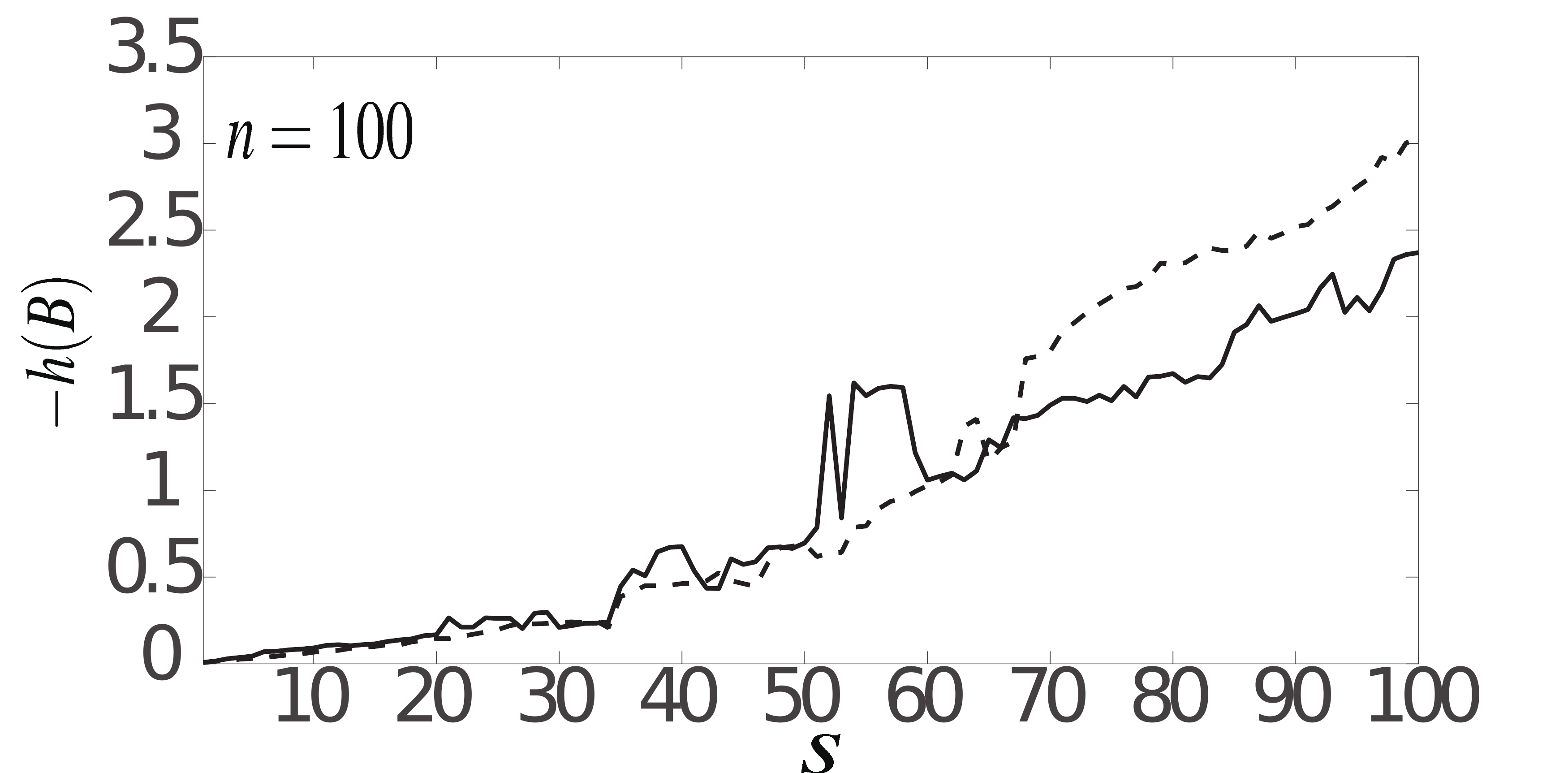}\\
\includegraphics[width = 4.5cm, height = 3cm]{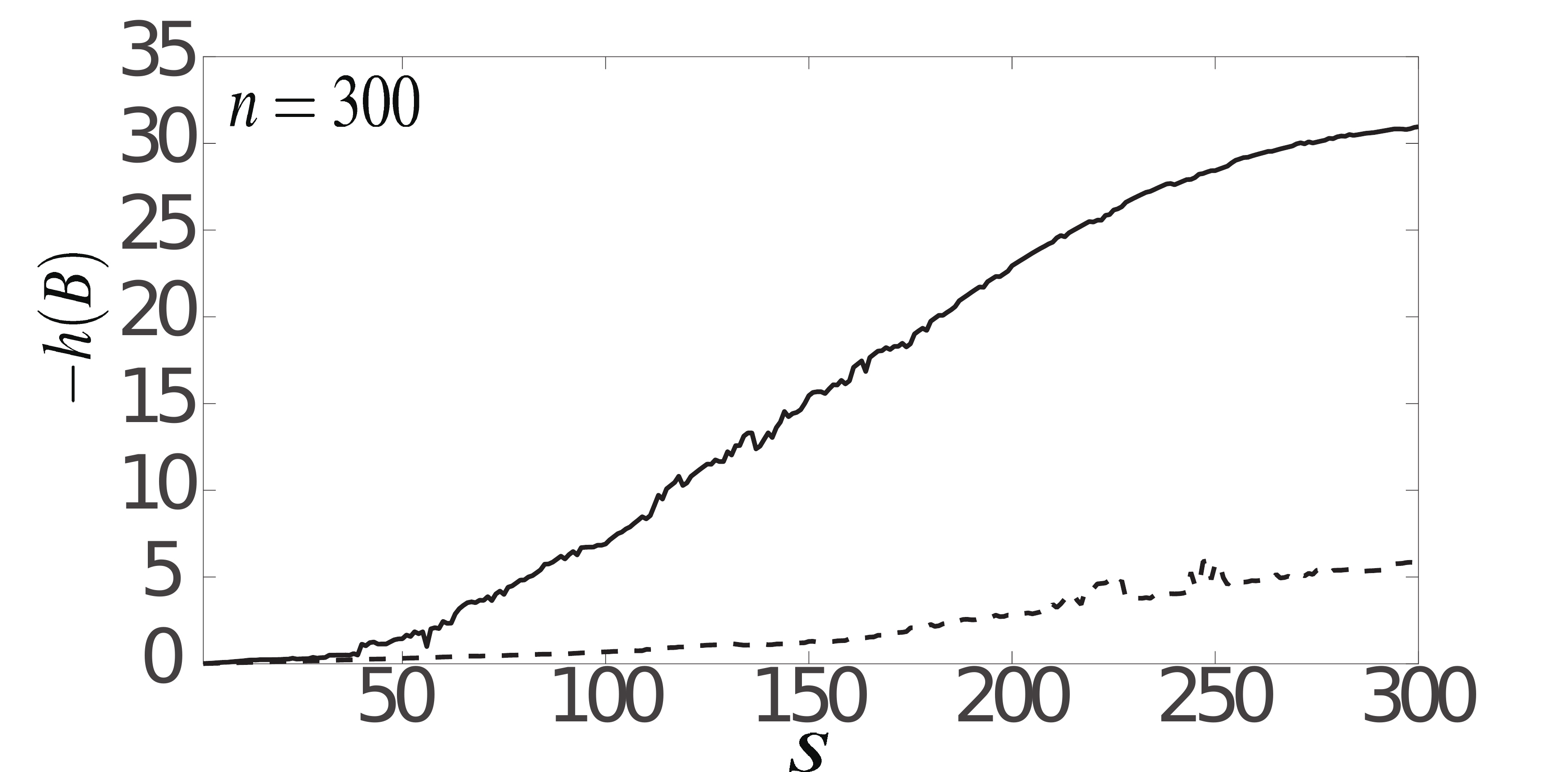}
\end{center}
\vspace{-1mm}
\end{minipage}
\end{tabular}
\caption{The relations between controllability index $-h(B)$ and sparsity parameter $s$ in Problem 1. The top-left and top-right figures show the simulation results in $n=50$ and $n=100$, respectively. The bottom-left and bottom-right figures show the simulation results in $n=200$ and $n=300$, respectively. The solid and dashed lines indicate the results of the {\rm sprandn} and Watts--Strogats cases, respectively.} \label{Fig1}
\end{figure}

\begin{figure}[t]
 \begin{tabular}{cc}
\begin{minipage}{0.45\hsize}
\begin{center}
\includegraphics[width = 4.5cm, height = 2.8cm]{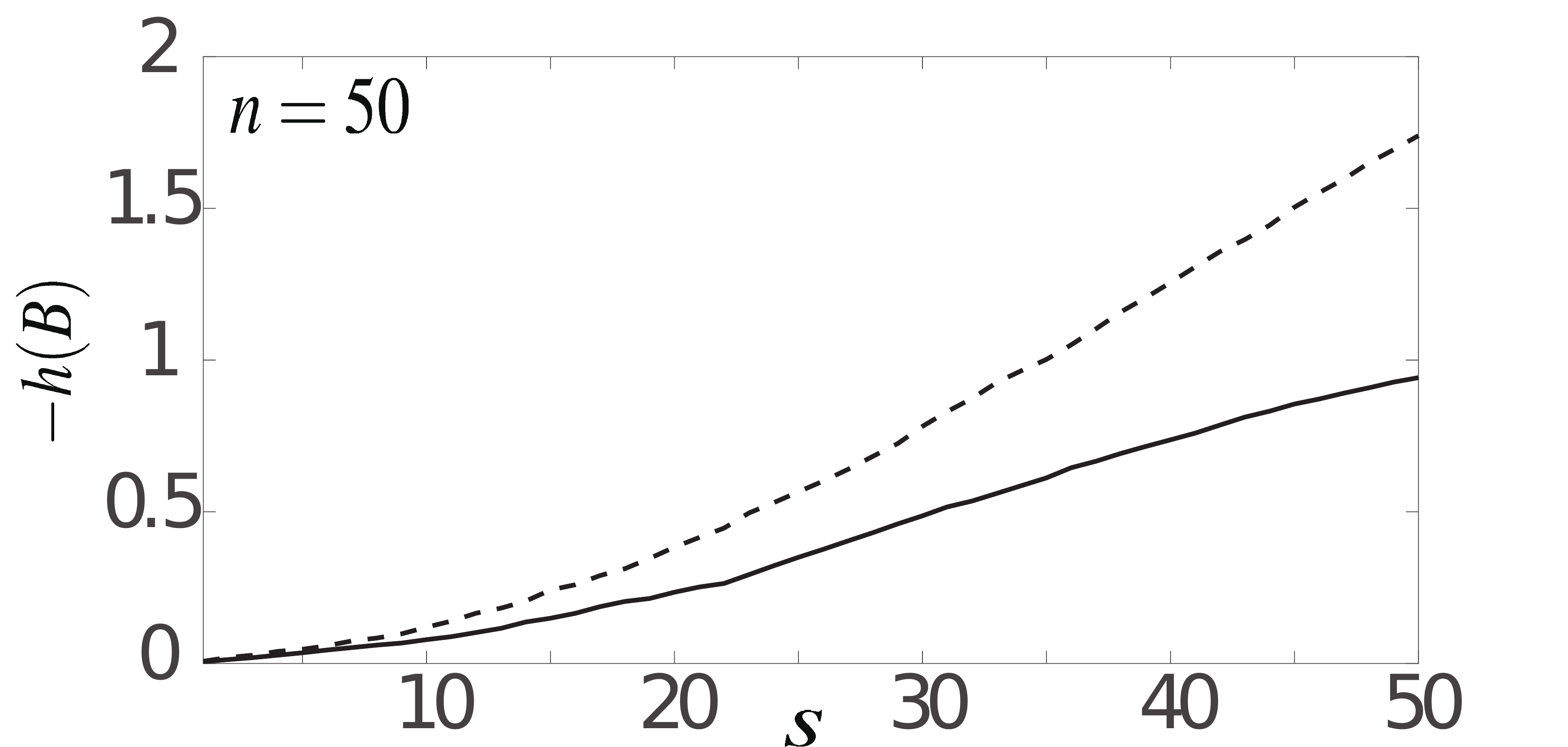}\\
\includegraphics[width = 4.5cm, height = 3cm]{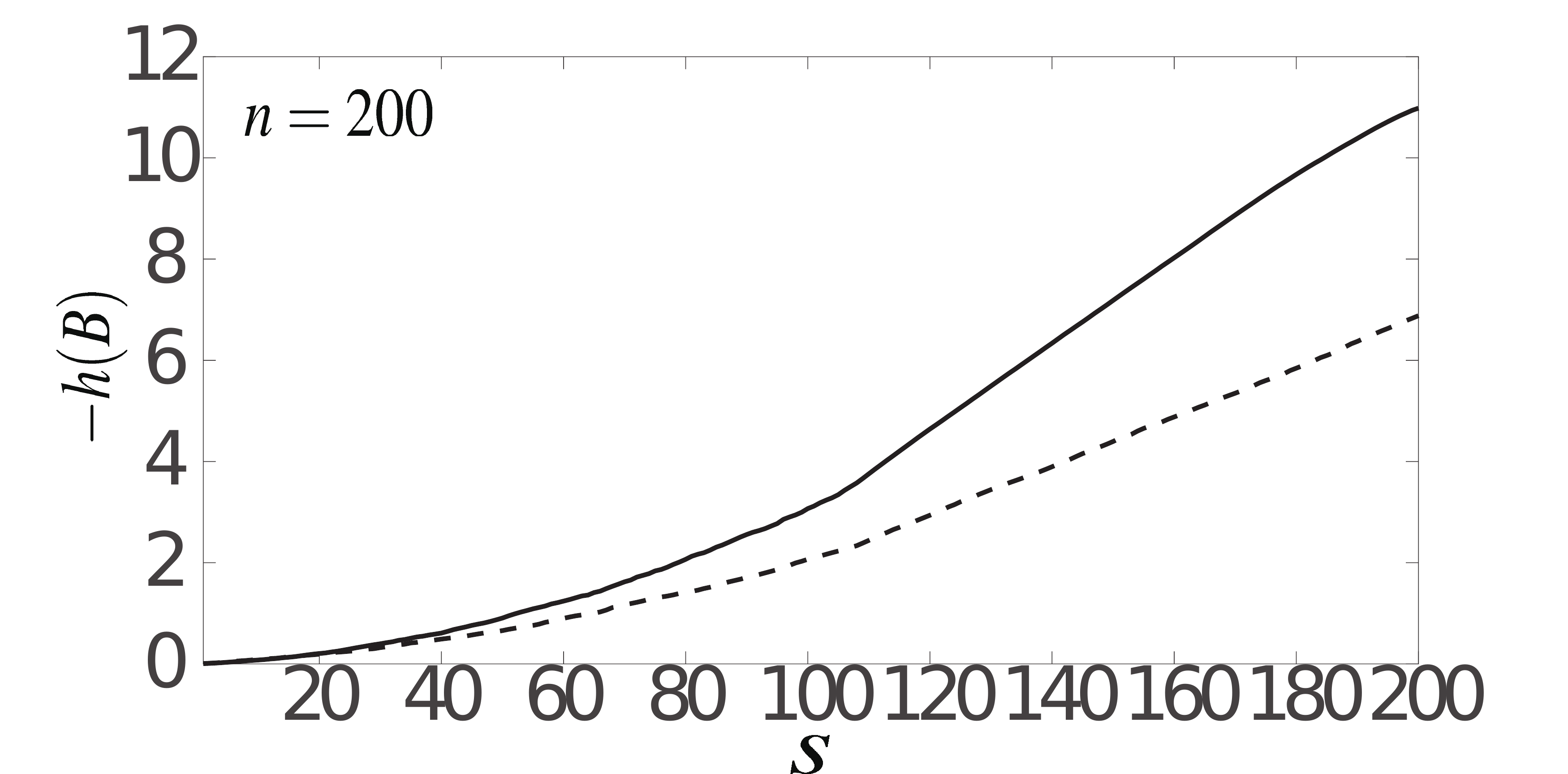}
\end{center}
\end{minipage} &
\begin{minipage}{0.45\hsize}
\begin{center}
\includegraphics[width = 4.5cm, height = 3cm]{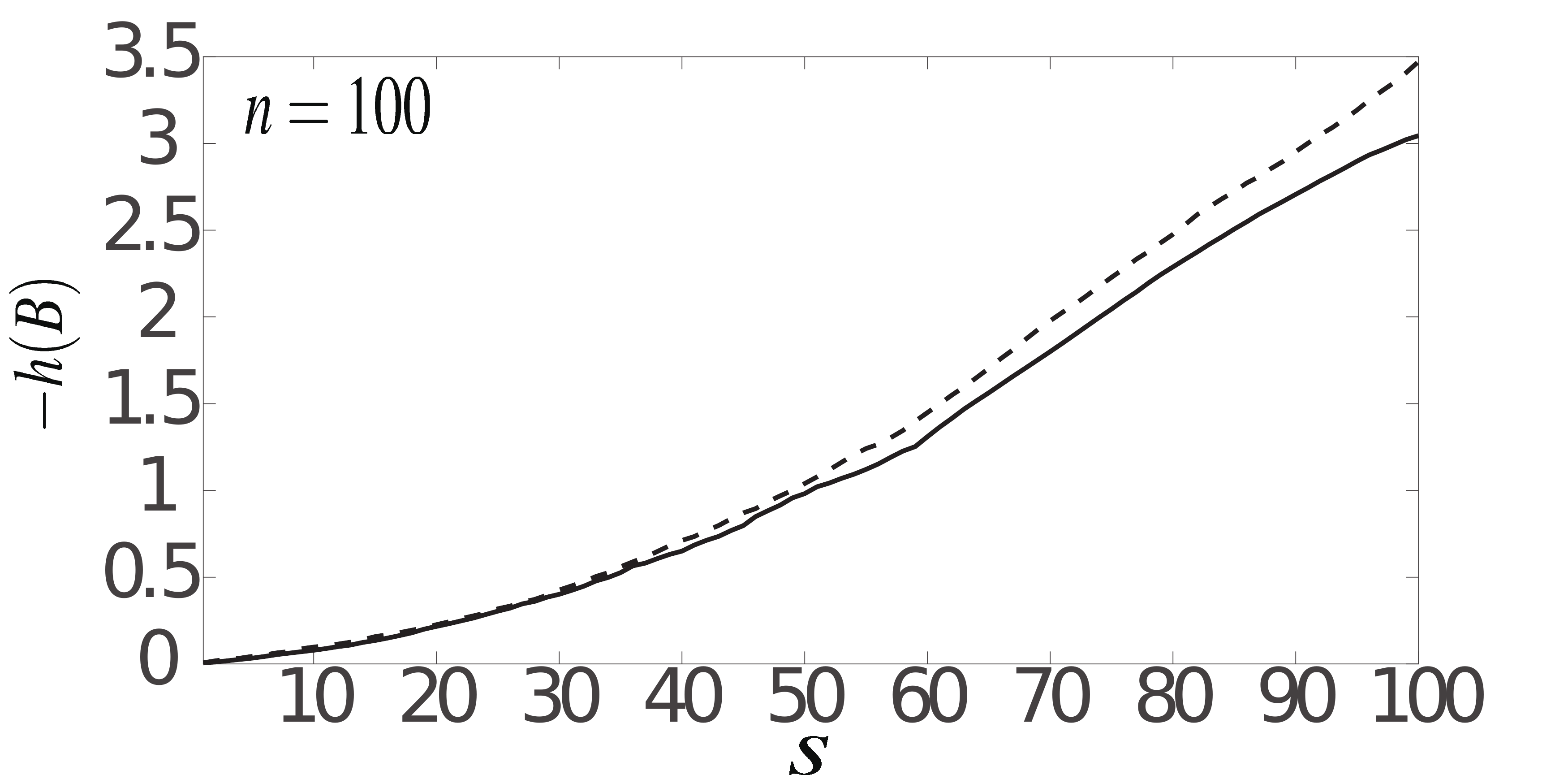}\\
\includegraphics[width = 4.5cm, height = 3cm]{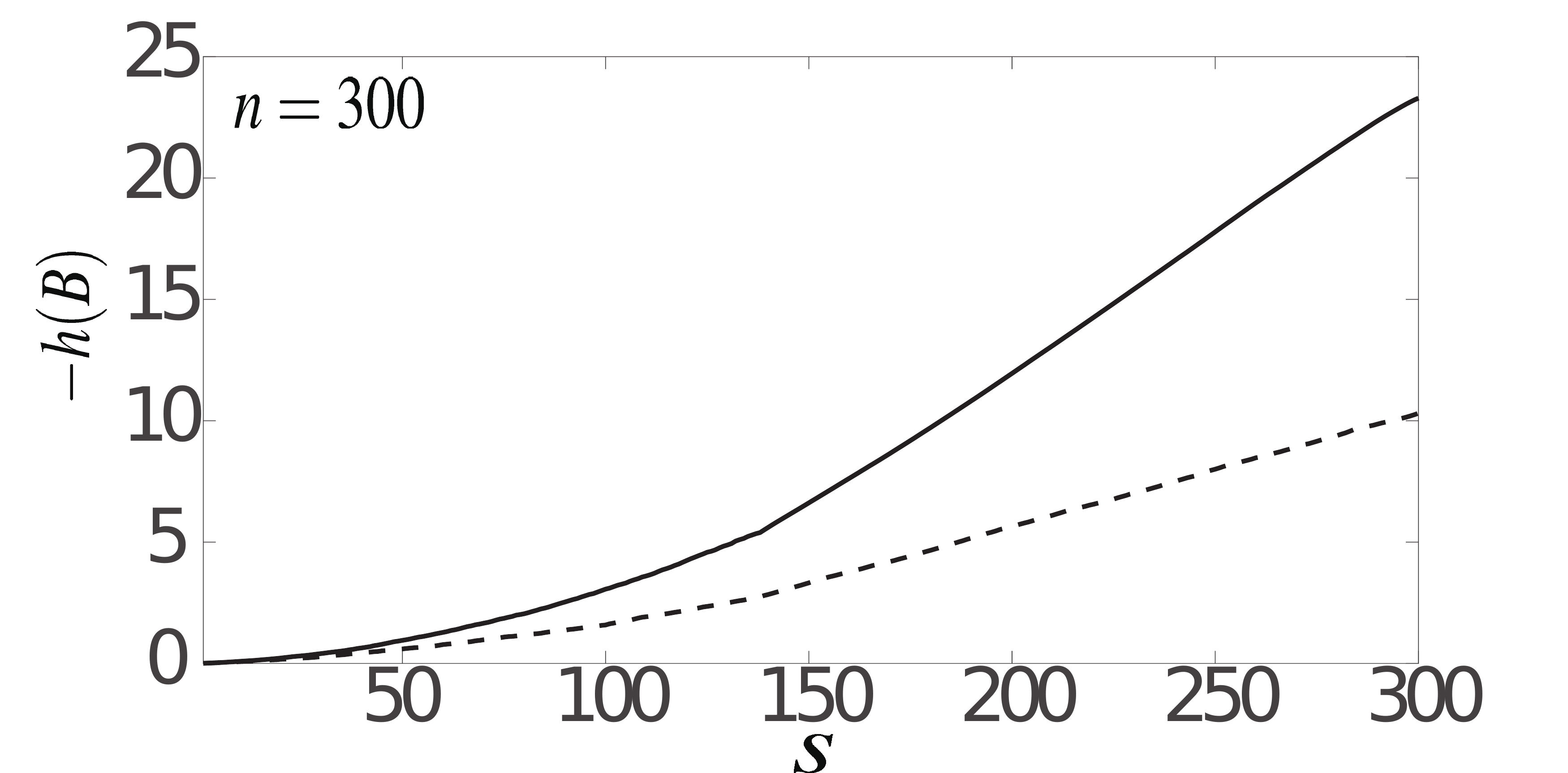}
\end{center}
\vspace{-1mm}
\end{minipage}
\end{tabular}
\caption{The relations between controllability index $-h(B)$ and sparsity parameter $s$ in Problem 2. The top-left and top-right figures show the simulation results in $n=50$ and $n=100$, respectively. The bottom-left and bottom-right figures show the simulation results in $n=200$ and $n=300$, respectively. The solid and dashed lines indicate the results of the {\rm sprand} and Watts--Strogats cases, respectively. } \label{Fig2}
\end{figure}

The following non-trivial results were obtained.
\begin{enumerate}
\item For Problems 1 and 2, the controllability characteristic changed as $s$ increased.
That is, the slope of $-h(B)$ varied for small and sufficiently large values of $s$.

\item For Problems 1 and 2, the controllability index $-h(B)$ did not saturate as $s$ increased.

\item 
In the case of $s=n$, $B$ generated by Algorithm 1 had $n$ non-zero elements for Problems 1 and 2.
That is, the controllability in terms of $-h(B)$ was maximized when all the states $x_i(t)$, $i=1,2,\ldots,n$ were stimulated by a common single input $u(t)\in {\bb R}$.
Moreover, although the elements of $B$ had values of only $1$ for Problem 2, as can be shown theoretically,
the corresponding elements in Problem 1 were vectors composed of non-trivial combinations of $1$ and $-1$.

\item For Problem 1 (Problem 2), although controllability in the {\rm sprandn} (sprand) cases were higher than that in the Watts--Strogats cases for $n=50$ and $n=100$; however,
the relations were reversed for $n=200$ and $n=300$.

\end{enumerate}

\begin{remark}
We confirmed that ${\rm rank}\, \mathcal{C}_T(B)<10$ for Problems 1 and 2 when $n=50$, $100$, $200$, and $300$.
That is, system \eqref{system} resulting from Algorithm 1 was not controllable.
Thus, $\lambda_{\min}(\mathcal{C}_T(B))=0$ and ${{\rm tr}(\mathcal{C}_T^{-1}(B))}$ could not be defined.
However, we could increase $-h(B)$.
This means that the system controllability can be increased on a low dimensional subspace,
and it is expected that the such a subspace is determined by the structure of $A$.
\end{remark}

\section{Conclusion} \label{Sec5}

We formulated two novel controllability maximization problems and developed a simple projected gradient method for solving the problems.
We proved that a sequence generated by our method has global convergence with locally linear convergence rate.
Moreover, the projections used in the proposed method were given explicitly.
Numerical experiments demonstrated the effectiveness of our method, and provided
non-trivial results.
In particular,
it is indicated that controllability characteristic changes as the parameter specifying sparsity increases, and the change rate appears to be dependent on a network structure.
The analysis of the change rates for various network structures would be considered in future work.


%


\section*{Acknowledgment}
This work was supported by Japan
Society for the Promotion of Science KAKENHI under Grant 19H04069. 



\ifCLASSOPTIONcaptionsoff
  \newpage
\fi



\bibliographystyle{IEEEtran}





%




\end{document}